\documentclass{amsproc}

\newtheorem{theorem}{Theorem}[section]
\newtheorem{lemma}[theorem]{Lemma}
\newtheorem{proposition}[theorem]{proposition}

\theoremstyle{definition}
\newtheorem{definition}[theorem]{Definition}

\theoremstyle{remark}
\newtheorem{remark}[theorem]{Remark}

\numberwithin{equation}{section}

\begin{document}

\title[Higher-order regularity of solutions]{Higher-order regularity of solutions to the large scale moist atmosphere system}


\author{Shenyang Tan}
\address{$^{1}$School of Mathematics and Statistics,
Nanjing University of Information Science
 and Technology, Nanjing 210044, China\\
 $^{2}$Taizhou Institute of Sci. $\&$ Tech. NJUST, Taizhou 225300, China}
\curraddr{}
\email{tanshenyang@njust.edu.cn}
\thanks{}

\author{Wenjun Liu}
\address{School of Mathematics and Statistics,
Nanjing University of Information Science
 and Technology, Nanjing 210044, China}
\curraddr{}
\email{wjliu@nuist.edu.cn(Corresponding author)}
\thanks{This work was supported by the Natural Science Foundation  of China (No. 11771216), the Key Research and Development Program of Jiangsu Province (Social Development) (No.  BE2019725), the Qing Lan Project of Jiangsu Province and
Postgraduate Research and Practice Innovation Program of Jiangsu Province (No. KYCX21\_0930).}



\date{}

\dedicatory{}

\begin{abstract}
In this paper, we study the higher-order regularity of solutions to the large scale moist atmosphere system through the way of $p$-strong solutions. On the basis of the well-posedness results of strong solutions, we first improve the regularity of solutions in the vertical direction, and then improve the regularity in the horizontal direction. Thus we obtain the well-posedness of solutions in $H^{2}$ space.
\end{abstract}

\maketitle

\section{Introduction }
In a noninertial coordinate system, under the $(x,y,p)$ coordinates, the system of large-scale moist atmosphere is formed by coupling the primitive equations
\begin{align}
\label{pe1}
&\partial_{t}\textbf{\emph{v}}+(\textbf{\emph{v}}\cdot\nabla)\textbf{\emph{v}}
+\omega\partial_{p}\textbf{\emph{v}}+\nabla\Phi+f \textbf{\emph{v}}^{\bot}+\mathcal{A}_{\textbf{\emph{v}}}\textbf{\emph{v}}=f_{\textbf{\emph{v}}},\\
 \label{pe2}&\partial_{p}\Phi=-\frac{RT}{p},\\
\label{pe3}&\nabla\cdot \textbf{\emph{v}}+\partial_{p}\omega=0,\\
&\partial_{t}T+\textbf{\emph{v}}\cdot\nabla T+\omega\partial_{p}T-\frac{RT}{c_{p}p}\omega+\mathcal{A}_{T}T=f_{T},\label{pe4}
\end{align}
and the conservation equation of water vapor in the air
\begin{align}\label{qe}
\partial_{t}q+\textbf{\emph{v}}\cdot\nabla q+\omega\partial_{p}q+\mathcal{A}_{q}q=f_{q}.
\end{align}
Here the unknown functions are $\textbf{\emph{v}},\omega,T,q$, where $\textbf{\emph{v}}=(v_{1},v_{2})$ is the horizontal velocity vector, $\textbf{\emph{v}}^{\bot}=(-v_{2},v_{1})$, $\omega$ is the vertical velocity under $(x,y,p)$ coordinates, $\nabla=(\partial_{x},\partial_{y})$, $T$ is the temperature, $\Phi$ is the geopotential, $f$ is the Coriolis force parameter, $R$ is the gas constants for dry air, $c_{p}$ is the specific heat of air at constant pressure, $p$ is the pressure, $f_{\textbf{\emph{v}}},f_{T},f_{q}$ are the source terms, $\mathcal{A}_{v},\mathcal{A}_{T},\mathcal{A}_{q}$ are the viscosity terms and
\begin{align*}
\mathcal{A}_{\ast}=-\mu_{\ast}\Delta-\nu_{\ast}\frac{\partial}{\partial p}\left(\left(\frac{gp}{R\bar{\theta}}\right)^{2}\frac{\partial}{\partial p}\right),\ast=\textbf{\emph{v}},T,q.
\end{align*}

   The primitive equations (\ref{pe1})-(\ref{pe4}) is extensively used in the study of atmosphere science. It was proposed systematically from the mathematical point of view by Lions, Temam and Wang \cite{Lions} in 1992. Since then, a lot of research have been done about this atmosphere dynamic system, see \cite{CaoTiti,CaoTiti2,CaoTiti3,Gao1,Guo3,Ju,Ju2,kukavica1,kukavica2,You,Zhou} and references therein.  The most noteworthy work is that Cao, Titi,
and Wang [2] completely obtained the global well-posedness of strong solution in $H^{1}$ space. On the higher-order regularity of solutions to the primitive equations, Petcu and Wirosoetisno \cite{Petcu1} proved the local existence of very regular solutions, up to $C^{\infty}$-regularity. Then based on the excellent work in \cite{CaoTiti}, Petcu \cite{Petcu2,Temam-ziane} extended the existence time interval from locally to globally.

   As to the moist atmospheric system (\ref{pe1})-(\ref{qe}), it was proposed mathematically by Guo and Huang in \cite{Guo1,Guo2}. They obtained the well-posedness of solutions in $H^{1}$ space. In a series works of Coti Zelati and Temam et al. \cite{Bousquet,Zelati2,Cao1,Zelati,Zelati3,Temam-wu}, they further divided the moist atmosphere into water vapor, rain water and cloud condensates, and introduced the Heaviside function in source terms to describe the transformation between different phases. Hittmeir and Klein et al. \cite{Hittmeir,Hittmeir2017} also studied the multi-phase moist atmospheric system, in which the microphysical factors in source terms were also considered.
   Inspired by the works on moist atmospheric system, in \cite{TanLiu} we also study a multi-phase moist atmospheric system, in which the velocity is dominated by the primitive equations.

In this paper, we focus on the mathematical method to study the $H^{2}$ regularity of solutions to the moist atmospheric system, but we do not consider the
microphysical factors during the phase transition of moisture atmosphere. So we still consider the atmospheric system (\ref{pe1})-(\ref{qe}). We use the so-called $p$-strong solution method to study the $H^{2}$ regularity of solutions to the moist atmospheric system, which is different from the method in \cite{Petcu2,Temam-ziane}.

The rest of this paper is organized as follows.
In Section 2, we give the mathematical formulation of moist atmospheric system and the main results about the well-posedness of $p$-strong and $H^{2}$ solutions. In Section 3, we give the a priori estimates of $p$-strong solutions and the proof of Theorem \ref{theorem1}.
In Section 4, we improve the regularity of solutions in the horizontal direction, and then give the proof of Theorem \ref{theorem2}.

\section{Mathematical formulation and main results}
\subsection{Mathematical formulation}
Reexamining the temperature equation (\ref{pe4}), we can find that the term $\omega RT/c_{p}p$ is a troublesome term. In fact, we can eliminate this term by introducing the potential temperature
\begin{align*}
\theta=T\left(\frac{p_{0}}{p}\right)^{R/c_{p}}-\theta_{h},
\end{align*}
where $\theta_{h}$ is a reference temperature satisfying $\theta_{h},\partial\theta_{h}/\partial p\in L^{\infty}\left((0,t)\times\mathcal{M}\right)$.
Then the equation for $T$ becomes
\begin{align*}
\partial_{t}\theta+\textbf{\emph{v}}\cdot\nabla\theta+\omega\partial_{p}\theta+\mathcal{A}_{\theta}\theta=f_{\theta},
\end{align*}
where
 \begin{align*}
\mathcal{A}_{\theta}=-\mu_{\theta}\Delta-\nu_{\theta}\left(\frac{p_{0}}{p}\right)^{R/c_{p}}
\frac{\partial}{\partial p}\left(\left(\frac{gp}{R\bar{\theta}}\right)^{2}\frac{\partial}{\partial p}\left(\frac{p_{0}}{p}\right)^{R/c_{p}}\right).
\end{align*}

In summary, we mainly consider the following system which is composed of the primitive equations and the humidity equation:
\begin{equation}\label{e1}
\partial_{t}\textbf{\emph{v}}+\mathcal{A}_{\textbf{\emph{v}}}\textbf{\emph{v}}+(\textbf{\emph{v}}\cdot\nabla)\textbf{\emph{v}}+\omega\partial_{p}\textbf{\emph{v}}+\nabla\Phi+f \textbf{\emph{v}}^{\bot}=f_{\textbf{\emph{v}}},
\end{equation}
\begin{equation}\label{e2}
\partial_{p}\Phi+\frac{RT}{p}=0,
\end{equation}
\begin{equation}\label{e3}
\nabla\cdot \textbf{\emph{v}}+\partial_{p}\omega=0,
\end{equation}
\begin{equation}\label{e4}
\partial_{t}\theta+\mathcal{A}_{\theta}\theta+\textbf{\emph{v}}\cdot\nabla\theta+\omega\partial_{p}\theta= f_{\theta},
\end{equation}
\begin{equation}\label{e5}
\partial_{t}q+\mathcal{A}_{q}q+\textbf{\emph{v}}\cdot\nabla q+\omega\partial_{p}q= f_{q}.
\end{equation}

Let $\mathcal{M}'=(0,1)\times (0,1)$ be a rectangular domain in $\mathbb{R}^{2}$, and $\mathcal{M}=\mathcal{M}'\times (p_{0},p_{1})$ be a cylinder in $\mathbb{R}^{3}$ under $(x,y,p)$ coordinate, where $0<p_{0}<p_{1}$. The boundary of $\mathcal{M}$ is composed of $\Gamma_{i}, \Gamma_{u}, \Gamma_{l}$, where
\begin{align*}
\Gamma_{i}&=\{(x,y,p)\in\bar{\mathcal{M}}:p=p_{1}\},\\
\Gamma_{u}&=\{(x,y,p)\in\bar{\mathcal{M}}:p=p_{0}\},\\
\Gamma_{l}&=\{(x,y,p)\in\bar{\mathcal{M}}:(x,y)\in\partial\mathcal{M}',p_{0}\leq p\leq p_{1}\}.
\end{align*}

For mathematical convenience, we consider the following periodic boundary conditions:
$\textbf{\emph{v}},\omega$, $\theta,\zeta,P_{v}$ are periodic in $x,y$ with period 1, and are periodic in $p$ with period $p_{1}-p_{0}$. If we set $\tilde{p}=p-(p_{0}+p_{1})/2$, then $\textbf{\emph{v}}(\tilde{p}),\zeta(\tilde{p})$ are even in $\tilde{p}$. $\omega(\tilde{p}),\theta(\tilde{p})$ are odd in $\tilde{p}$. In addition, we suppose that $\partial_{p}\textbf{\emph{v}}=\omega=0$ when $p=p_{0}=p_{1}$.
The initial conditions are
\begin{align}
\textbf{\emph{v}}(x,y,p,0)=\textbf{\emph{v}}_{0}(x,y,p),\
\theta(x,y,p,0)=\theta_{0}(x,y,p),\
q(x,y,p,0)=q_{0}(x,y,p)\label{initial data3}.
\end{align}
We also assume that the initial data satisfy the same periodicity.

Throughout this paper, for simplicity, we denote by $H^{s}=H^{s}(\mathcal{M})$ the classical Sobolev spaces of order $s$ on $\mathcal{M}$, and by $L^{p}=L^{p}(\mathcal{M}),1\leq p\leq\infty$ the classical $L^{p}$ Lebesgue space with norm $\|\cdot\|_{L^{p}}$. We also use the abbreviation $\|f\|_{w}$ for $\|(gp/R\bar{\theta})f\|_{L^{2}}$.
It is easy to verify that $\|f\|_{w}$ is equivalent to $\|f\|_{L^{2}}$.
Considering the free divergence property of velocity $\textbf{\emph{v}}$, we introduce the space
\begin{align*}
\mathcal{V}&=\left\{\textbf{\emph{v}}\in C^{\infty}(\mathcal{M};\mathbb{R}^{2}):\nabla\cdot\int_{p_{0}}^{p_{1}}
\textbf{\emph{v}}(x,y,p')dp'=0,\textbf{\emph{v}}\ \text{satisfies}\ (\ref{e1})\right\},\\
\mathbb{H}&=\text{The closure of}\ \mathcal{V}\ \text{with respect to the norm of}\ (L^{2})^{2}, \\
\mathbb{V}&=\text{The closure of}\ \mathcal{V}\ \text{with respect to the norm of}\ (H^{1})^{2}.
\end{align*}
\begin{definition}
Let $\textbf{\emph{v}}_{0}\in\mathbb{V},\theta_{0}, q_{0}\in H^{1}$, $f_{\textbf{\emph{v}}},f_{\theta},f_{q}\in H^{1}$. A solution $(\textbf{\emph{v}},\theta,q)$ to equations (\ref{e1})-(\ref{e5}) with the periodic boundary conditions and initial condition (\ref{initial data3}) is called a global strong solution if for any $t_{1}\in (0,\infty)$,
\begin{align*}
&(\textbf{\emph{v}},\theta,q)\in C(0,t_{1};\mathbb{V}\times (H^{1})^{2})\cap L^{2}(0,t_{1}; (H^{2})^{4}),\
\partial_{t}(\textbf{\emph{v}},\theta,q)\in L^{2}(0,t_{1};\mathbb{H}\times (L^{2})^{2}).
\end{align*}
\end{definition}

\begin{definition}
Let $\textbf{\emph{v}}_{0},\partial_{p}\textbf{\emph{v}}_{0}\in\mathbb{V},\theta_{0},\partial_{p}\theta_{0}, q_{0},\partial_{p}q_{0}\in H^{1}$, $\partial_{p}f_{\textbf{\emph{v}}},\partial_{p}f_{\theta},\partial_{p}f_{q}\in L^{2}$. A strong solution $(\textbf{\emph{v}},\theta,q)$ to equations (\ref{e1})-(\ref{e5}) with the periodic boundary conditions and initial condition (\ref{initial data3}) is called a global $p$-strong solution if for any $t_{1}\in (0,\infty)$,
\begin{align*}
\left(\partial_{p}\textbf{\emph{v}},\partial_{p}\theta,\partial_{p}q\right)\in C\left(0,t_{1};\mathbb{V}\times (H^{1})^{2}\right)\cap L^{2}\left(0,t_{1};(H^{2})^{4}\right).
\end{align*}
\end{definition}

We state our main results on the global existence of $p$-strong and $H^{2}$ solutions to the moist atmospheric system as follows.
\begin{theorem}\label{theorem1}
Let $\textbf{v}_{0}, \partial_{p}\textbf{v}_{0}\in\mathbb{V}$, $\theta_{0},\partial_{p}\theta_{0},q_{0},\partial_{p}q_{0}\in H^{1},$ $\partial_{p}f_{\textbf{v}},\partial_{p}f_{\theta},\partial_{p}f_{q}\in L^{2}$. Then equations (\ref{e1})-(\ref{e5}) with the periodic boundary conditions and initial condition (\ref{initial data3}) has a unique global $p$-strong solution for any $t_{1}\in (0,\infty)$.
\end{theorem}

\begin{theorem}\label{theorem2}
Let $\left(\textbf{v}_{0},\theta_{0},q_{0}\right)\in\mathbb{V}\cap(H^{2})^{2}\times (H^{2})^{2}$, $f_{\textbf{v}},f_{\theta},f_{q}\in H^{1}$. Then equations (\ref{e1})-(\ref{e5}) with the periodic boundary conditions and initial condition (\ref{initial data3}) has a unique solution $(\textbf{v},\theta,q)$ satisfying that
\begin{align*}
&\left(\textbf{v},\theta,q\right)\in C\left(0,t_{1};\mathbb{V}\cap(H^{2})^{2}\times (H^{2})^{2}\right)\cap L^{2}\left(0,t_{1};(H^{3})^{4}\right),\nonumber\\
&\left(\partial_{t}\textbf{v},\partial_{t}\theta,\partial_{t}q\right)\in L^{2}\left(0,t_{1};\mathbb{V}\times(H^{1})^{2}\right),
\end{align*}
for any $t_{1}\in (0,\infty)$.
\end{theorem}
\begin{remark}
Comparing the result in Theorem \ref{theorem1} with the well-posedness result of strong solution in Lemma \ref{H1 strong existence} below, we can find that if we improve the regularity of initial data and source terms in the $p$ direction, then the regularity of the strong solution of equation (\ref{e1})-(\ref{e5}) in the $p$ direction will also be improved accordingly.
Comparing results in Theorem \ref{theorem1} with Theorem \ref{theorem2}, we know that
when we continue to improve the regularity of initial data and source terms of $p$-strong solution in the horizontal direction, the solution will attain the $H^{2}$ regularity.
Based on the idea of $p$-strong solution, we can successively improve the regularity of solutions up to $C^{\infty}$.
\end{remark}
\begin{remark}
Using the a priori estimates in this paper, combining the uniform Gronwall Lemma, through similar arguments as in \cite{Ju,You}, one can obtain the existence of global attractor in $H^{2}$ space.
\end{remark}
\begin{remark}
In this paper, we consider the $H^{2}$-regularity of solutions to the moist atmospheric system with the periodic boundary condition. In fact, combining methods in this paper with the method in \cite{Hittmeir}, we can also deal with the case with physical boundary condition, which is more in line with the actual situation.
\end{remark}
\section{The well-posedness of $p$-strong solutions }
In this section, we consider the well-posedness of $p$-strong solutions. We first give a well-posedness result about the strong solution to equations (\ref{e1})-(\ref{e5}). Through a similar proof as in \cite{Zelati,Guo2}, we can get the following result:
\begin{lemma}\label{H1 strong existence}
Let $f_{\textbf{v}},f_{\theta},f_{q}\in L^{2}$, $\textbf{v}_{0}\in\mathbb{V},\theta_{0}, q_{0}\in H^{1}$. Then equations (\ref{e1})-(\ref{e5}) associated with the initial data (\ref{initial data3}) has a unique global strong solution $(\textbf{v},\theta,q)$ satisfying that
\begin{align*}
(\textbf{v},\theta,q)\in C(0,t_{1};\mathbb{V}\times (H^{1})^{2})\cap L^{2}(0,t_{1}; (H^{2})^{4}),\ \partial_{t}(\textbf{v},\theta,q)\in L^{2}(0,t_{1};\mathbb{H}\times (L^{2})^{2}),
\end{align*}
for any $t_{1}\in (0,\infty)$.
\end{lemma}

The following lemma is useful during the a priori estimates.
\begin{lemma}(\cite[proposition 2.2]{CaoTiti6})\label{trilinear}
For any $f,g,h \in H^1$,
\begin{align*}
&\left| \int_{\mathcal{M}'} \int_{p_{0}}^{p_{1}} fdp\int_{p_{0}}^{p_{1}} ghdpdxdy \right|\\
\leq& C \|f\|_{L^2}^{\frac{1}{2}}\left(\|f\|_{L^{2}}^{\frac{1}{2}} + \|\nabla f\|_{L^{2}}^{\frac{1}{2}}\right)
\|g\|_{L^{2}} \|h\|_{L^{2}}^{\frac{1}{2}}\left(\|h\|_{L^{2}}^{\frac{1}{2}}+ \|\nabla h\|_{L^{2}}^{\frac{1}{2}}\right).
\end{align*}
\end{lemma}

Let us recall the Minkowski inequality in integral form which is useful in estimating for $\omega$:
Let $\Omega_{1}\in\mathbb{R}^{m_{1}}$ and $\Omega_{2}\in\mathbb{R}^{m_{2}}$ be two measurable sets, where $m_{1}$ and $m_{2}$ are two positive integers. Suppose that $f(x,y)$ is measurable over $\Omega_{1}\times\Omega_{2}$. Then for any $p\geq 1$,
\begin{align}
\left[\int_{\Omega_{1}}\left(\int_{\Omega_{2}}|f(x,y)|d\Omega_{2}\right)^{p}d\Omega_{1}\right]^{\frac{1}{p}} \leq\int_{\Omega_{2}}\left(\int_{\Omega_{1}}|f(x,y)|^{p}d\Omega_{1}\right)^{\frac{1}{p}}d\Omega_{2}.
\end{align}

We start with the a priori estimate for $p$-strong solution $\textbf{\emph{v}}$ to equation (\ref{e1}).
\begin{proposition}\label{pro1}
Let $\textbf{v}_{0}, \partial_{p}\textbf{v}_{0}\in\mathbb{V}$, $\partial_{p}f_{\textbf{v}}\in L^{2}$. Then for any $t_{1}\in (0,\infty)$, the strong solution $\textbf{v}$ to equation (\ref{e1}) satisfies that
\begin{align*}
\partial_{p}\textbf{v}\in L^{\infty}\left(0,t_{1};\mathbb{V}\right)\cap L^{2}\left(0,t_{1};H^{2}\right).
\end{align*}

\end{proposition}
\begin{proof}
 Differentiating equation (\ref{e1}) with respect to $p$, we have
\begin{align*}
\partial_{t}\textbf{\emph{v}}_{p}+(\textbf{\emph{v}}_{p}\cdot\nabla)\textbf{\emph{v}}
+(\textbf{\emph{v}}\cdot\nabla)\textbf{\emph{v}}_{p}+(\nabla\cdot\textbf{\emph{v}})\textbf{\emph{v}}_{p}
+\omega\partial_{p}\textbf{\emph{v}}_{p}+f\textbf{\emph{v}}_{p}^{\bot}-
\frac{R}{p}\nabla T+\partial_{p}(\mathcal{A}_{\textbf{\emph{v}}}\textbf{\emph{v}})=\partial_{p}f_{\textbf{\emph{v}}},
\end{align*}
where we have denoted $\partial_{p}\textbf{\emph{v}}$ by $\textbf{\emph{v}}_{p}$ for simplicity. Taking the inner product of  the above equation with $-(\Delta\textbf{\emph{v}}_{p}+\partial_{p}^{2}\textbf{\emph{v}}_{p})$ in $L^{2}$ space, we can get that
\begin{align*}
&-\int_{\mathcal{M}}\partial_{t}\textbf{\emph{v}}_{p}\cdot(\Delta\textbf{\emph{v}}_{p}+\partial_{p}^{2}\textbf{\emph{v}}_{p})d\mathcal{M}
-\int_{\mathcal{M}}\partial_{p}(\mathcal{A}_{\textbf{\emph{v}}}\textbf{\emph{v}})\cdot(\Delta\textbf{\emph{v}}_{p}+\partial_{p}^{2}\textbf{\emph{v}}_{p})d\mathcal{M}\nonumber\\
=&\int_{\mathcal{M}}[(\textbf{\emph{v}}_{p}\cdot\nabla)\textbf{\emph{v}}+(\nabla\cdot\textbf{\emph{v}})\textbf{\emph{v}}_{p}]\cdot(\Delta\textbf{\emph{v}}_{p}+\partial_{p}^{2}\textbf{\emph{v}}_{p})d\mathcal{M}
+\int_{\mathcal{M}}(\textbf{\emph{v}}\cdot\nabla)\textbf{\emph{v}}_{p}\cdot(\Delta\textbf{\emph{v}}_{p}+\partial_{p}^{2}\textbf{\emph{v}}_{p})d\mathcal{M}
+\\
&\int_{\mathcal{M}}\omega\partial_{p}\textbf{\emph{v}}_{p}\cdot(\Delta\textbf{\emph{v}}_{p}+\partial_{p}^{2}\textbf{\emph{v}}_{p})d\mathcal{M}
+\int_{\mathcal{M}}(f\textbf{\emph{v}}_{p}^{\bot}-\frac{R}{p}\nabla T-\partial_{p}f_{\textbf{\emph{v}}})\cdot(\Delta\textbf{\emph{v}}_{p}+\partial_{p}^{2}\textbf{\emph{v}}_{p})d\mathcal{M}.
\end{align*}
It follows from integration by parts that
\begin{align*}
-\int_{\mathcal{M}}\partial_{t}\textbf{\emph{v}}_{p}\cdot(\Delta\textbf{\emph{v}}_{p}+\partial_{p}^{2}\textbf{\emph{v}}_{p})d\mathcal{M}
=\frac{1}{2}\frac{d}{dt}(\|\nabla\textbf{\emph{v}}_{p}\|_{L^{2}}^{2}
+\|\partial_{p}\textbf{\emph{v}}_{p}\|_{L^{2}}^{2}),
\end{align*}
and
\begin{align*}
&-\int_{\mathcal{M}}\partial_{p}(\mathcal{A}_{\textbf{\emph{v}}}\textbf{\emph{v}})\cdot(\Delta\textbf{\emph{v}}_{p}+\partial_{p}^{2}\textbf{\emph{v}}_{p})d\mathcal{M}\nonumber\\
=&\mu_{v}\|\Delta\textbf{\emph{v}}_{p}\|_{L^{2}}^{2}+\nu_{v}\|\nabla\partial_{p}\textbf{\emph{v}}_{p}\|_{w}^{2}
+2\nu_{v}\int_{\mathcal{M}}\left(\frac{g}{R\bar{\theta}}\right)^{2}p\nabla\textbf{\emph{v}}_{p}\cdot\nabla\partial_{p}\textbf{\emph{v}}_{p}d\mathcal{M}
+\mu_{v}\|\nabla\partial_{p}\textbf{\emph{v}}_{p}\|_{L^{2}}^{2}\nonumber\\
&+\nu_{v}\|\partial_{p}^{2}\textbf{\emph{v}}_{p}\|_{w}^{2}
+2\nu_{v}\int_{\mathcal{M}}\left(\frac{g}{R\bar{\theta}}\right)^{2}p\partial_{p}\textbf{\emph{v}}_{p}\cdot\partial_{p}^{2}\textbf{\emph{v}}_{p}d\mathcal{M}
+\nu_{v}\int_{\mathcal{M}}\left(\frac{g}{R\bar{\theta}}\right)^{2}\textbf{\emph{v}}_{p}\cdot\partial_{p}^{2}\textbf{\emph{v}}_{p}d\mathcal{M}\nonumber\\
\geq&\frac{3\mu_{1}}{4}\|\textbf{\emph{v}}_{p}\|_{H^{2}}^{2}
-C\left(\|\nabla\textbf{\emph{v}}_{p}\|_{L^{2}}^{2}
+\|\partial_{p}\textbf{\emph{v}}_{p}\|_{L^{2}}^{2}+\|\textbf{\emph{v}}_{p}\|_{L^{2}}^{2}\right),
\end{align*}
where $\mu_{1}=\rm{\min}\{\mu_{\emph{v}},\nu_{\emph{v}}\}$ and we have used H\"older's inequality and Young's inequality in the last step.
Utilizing the H\"older inequality, we have
\begin{align*}
&\int_{\mathcal{M}}(\textbf{\emph{v}}_{p}\cdot\nabla)\textbf{\emph{v}}\cdot\Delta\textbf{\emph{v}}_{p}d\mathcal{M}
+\int_{\mathcal{M}}(\nabla\cdot\textbf{\emph{v}})\textbf{\emph{v}}_{p}\cdot\Delta\textbf{\emph{v}}_{p}d\mathcal{M}
\nonumber\\
\leq &C\int_{\mathcal{M}}|\textbf{\emph{v}}_{p}||\nabla\textbf{\emph{v}}||\Delta\textbf{\emph{v}}_{p}|d\mathcal{M}
\leq C\|\nabla\textbf{\emph{v}}\|_{L^{4}}\|\textbf{\emph{v}}_{p}\|_{L^{4}}\|\Delta\textbf{\emph{v}}_{p}\|_{L^{2}}\nonumber\\
\leq& C\|\nabla\textbf{\emph{v}}\|_{L^{2}}^{\frac{1}{4}}\|\nabla\textbf{\emph{v}}\|_{H^{1}}^{\frac{3}{4}}\|\textbf{\emph{v}}_{p}\|_{L^{2}}^{\frac{1}{4}}
\|\textbf{\emph{v}}_{p}\|_{H^{1}}^{\frac{3}{4}}\|\Delta\textbf{\emph{v}}_{p}\|_{L^{2}}
\leq C\|\textbf{\emph{v}}\|_{H^{1}}^{\frac{1}{2}}\|\textbf{\emph{v}}\|_{H^{2}}^{\frac{3}{4}}
\|\textbf{\emph{v}}_{p}\|_{H^{1}}^{\frac{3}{4}}\|\Delta\textbf{\emph{v}}_{p}\|_{L^{2}}\nonumber\\
\leq& C\|\textbf{\emph{v}}\|_{H^{1}}\|\textbf{\emph{v}}\|_{H^{2}}^{\frac{3}{2}}
\|\textbf{\emph{v}}_{p}\|_{H^{1}}^{\frac{3}{2}}+\frac{\mu_{v}}{32}\|\Delta\textbf{\emph{v}}_{p}\|_{L^{2}}^{2}
\leq C\|\textbf{\emph{v}}\|_{H^{2}}^{2}
\|\textbf{\emph{v}}_{p}\|_{H^{1}}^{2}+\|\textbf{\emph{v}}\|_{H^{1}}^{4}+\frac{\mu_{1}}{32}\|\textbf{\emph{v}}_{p}\|_{H^{2}}^{2},
\end{align*}
 where we have used the Gagliardo-Nirenberg-Sobolev inequality in the third step.
 Similarly,
\begin{align*}
\int_{\mathcal{M}}\left[(\textbf{\emph{v}}_{p}\cdot\nabla)\textbf{\emph{v}}
+(\nabla\cdot\textbf{\emph{v}})\textbf{\emph{v}}_{p}\right]\cdot\partial_{p}^{2}\textbf{\emph{v}}_{p}d\mathcal{M}
\leq C\|\textbf{\emph{v}}\|_{H^{2}}^{2}
\|\textbf{\emph{v}}_{p}\|_{H^{1}}^{2}+\|\textbf{\emph{v}}\|_{H^{1}}^{4}+\frac{\mu_{1}}{32}\|\textbf{\emph{v}}_{p}\|_{H^{2}}^{2},
\end{align*}
By the H\"older inequality and the Young inequality, we can infer that
\begin{align*}
&\int_{\mathcal{M}}(\textbf{\emph{v}}\cdot\nabla)\textbf{\emph{v}}_{p}\cdot(\Delta\textbf{\emph{v}}_{p}+\partial_{p}^{2}\textbf{\emph{v}}_{p})
d\mathcal{M}\nonumber\\
\leq& \|\textbf{\emph{v}}\|_{L^{\infty}}\|\nabla\textbf{\emph{v}}_{p}\|_{L^{2}}\|\textbf{\emph{v}}_{p}\|_{H^{2}}
\leq C\|\textbf{\emph{v}}\|_{H^{2}}^{2}\|\textbf{\emph{v}}_{p}\|_{H^{1}}^{2}+\frac{\mu_{1}}{32}\|\textbf{\emph{v}}_{p}\|_{H^{2}}^{2},
\end{align*}
where we have used the embedding relation $H^{2}\subset L^{\infty}$ in the last step.
Utilizing the inequality in Lemma \ref{trilinear}, we can deduce that
\begin{align*}
&\int_{\mathcal{M}}\omega\partial_{p}\textbf{\emph{v}}_{p}\cdot\Delta\textbf{\emph{v}}_{p}d\mathcal{M}
\leq \int_{\mathcal{M'}}\int_{p_{0}}^{p_{1}}|\nabla\textbf{\emph{v}}|dp\int_{p_{0}}^{p_{1}}|\partial_{p}\textbf{\emph{v}}_{p}||\Delta\textbf{\emph{v}}_{p}|dpd\mathcal{M'}\nonumber\\
\leq& C\|\nabla\textbf{\emph{v}}\|_{L^2}^{\frac{1}{2}}\left(\|\nabla\textbf{\emph{v}}\|_{L^{2}}^{\frac{1}{2}} + \|\Delta\textbf{\emph{v}}\|_{L^{2}}^{\frac{1}{2}}\right)
\|\Delta\textbf{\emph{v}}_{p}\|_{L^{2}} \|\partial_{p}\textbf{\emph{v}}_{p}\|_{L^{2}}^{\frac{1}{2}}\left(\|\partial_{p}\textbf{\emph{v}}_{p}\|_{L^{2}}^{\frac{1}{2}}+ \|\nabla\partial_{p}\textbf{\emph{v}}_{p}\|_{L^{2}}^{\frac{1}{2}}\right)\nonumber\\
\leq& C\|\textbf{\emph{v}}\|_{H^1}^{\frac{1}{2}}\|\textbf{\emph{v}}\|_{H^2}^{\frac{1}{2}}
\|\partial_{p}\textbf{\emph{v}}_{p}\|_{L^{2}}^{\frac{1}{2}}\|\textbf{\emph{v}}_{p}\|_{H^{2}}^{\frac{3}{2}}
\leq C\|\textbf{\emph{v}}\|_{H^1}^{2}\|\textbf{\emph{v}}\|_{H^2}^{2}
\|\textbf{\emph{v}}_{p}\|_{H^{1}}^{2}+\frac{\mu_{1}}{32}\|\textbf{\emph{v}}_{p}\|_{H^{2}}^{2}.
\end{align*}
Through a similar argument, we have
\begin{align*}
\int_{\mathcal{M}}\omega\partial_{p}\textbf{\emph{v}}_{p}\cdot\partial_{p}^{2}\textbf{\emph{v}}_{p}d\mathcal{M}
\leq C\|\textbf{\emph{v}}\|_{H^1}^{2}\|\textbf{\emph{v}}\|_{H^2}^{2}
\|\textbf{\emph{v}}_{p}\|_{H^{1}}^{2}+\frac{\mu_{1}}{32}\|\textbf{\emph{v}}_{p}\|_{H^{2}}^{2}.
\end{align*}
Using the H\"older inequality and the Young inequality, we have
\begin{align*}
&\int_{\mathcal{M}}[f\textbf{\emph{v}}_{p}^{\bot}-\frac{R}{p}\nabla T-\partial_{p}f_{\textbf{\emph{v}}}]\cdot(\Delta\textbf{\emph{v}}_{p}+\partial_{p}^{2}\textbf{\emph{v}}_{p})d\mathcal{M}\nonumber\\
\leq &C\left(\|\textbf{\emph{v}}\|_{H^{1}}^{2}+\|\theta\|_{H^{1}}^{2}+\|\partial_{p}f_{\textbf{\emph{v}}}\|_{L^{2}}^{2}\right)
+\frac{3\mu_{1}}{32}\|\textbf{\emph{v}}_{p}\|_{H^{2}}^{2}.
\end{align*}
Combining the above inequalities, we can deduce that
\begin{align}\label{008}
\frac{d}{dt}\|\textbf{\emph{v}}_{p}\|_{H^{1}}^{2}+\mu_{1}\|\textbf{\emph{v}}_{p}\|_{H^{2}}^{2}
\leq &C\left(1+\|\textbf{\emph{v}}\|_{H^{2}}^{2}+\|\textbf{\emph{v}}\|_{H^1}^{2}\|\textbf{\emph{v}}\|_{H^2}^{2} \right)\|\textbf{\emph{v}}_{p}\|_{H^{1}}^{2}\nonumber\\
&+C\left(\|\textbf{\emph{v}}\|_{H^{1}}^{2}+\|\textbf{\emph{v}}\|_{H^{1}}^{4}+\|\theta\|_{H^{1}}^{2}+\|\partial_{p}f_{\textbf{\emph{v}}}\|_{L^{2}}^{2}\right).
\end{align}

 Then using the Gronwall inequality and regularities of $\textbf{\emph{v}},\theta$ in Lemma \ref{H1 strong existence}, we   finish the proof.
\end{proof}

\begin{proposition}\label{pro2}
Let $\textbf{v}_{0}, \partial_{p}\textbf{v}_{0}\in\mathbb{V}$, $\theta_{0},\partial_{p}\theta_{0},\partial_{p}f_{\textbf{v}},\partial_{p}f_{\theta}\in L^{2}$. Then for any $t_{1}\in (0,\infty)$, the solution $\theta$ to equation (\ref{e4}) satisfies that
\begin{align*}
\partial_{p}\theta\in L^{\infty}\left(0,t_{1};H^{1}\right)\cap L^{2}\left(0,t_{1};H^{2}\right).
\end{align*}
\begin{proof}
Differentiating equation (\ref{e4}) with respect to $p$, we have
\begin{align}\label{b0}
\partial_{t}\theta_{p}+\textbf{\emph{v}}_{p}\cdot\nabla\theta+\textbf{\emph{v}}\cdot\nabla\theta_{p}
+(\nabla\cdot\textbf{\emph{v}})\theta_{p}+\omega\partial_{p}\theta_{p}+\partial_{p}\mathcal{A}_{\theta}\theta
=\partial_{p}f_{\theta},
\end{align}
where we have denoted $\partial_{p}\theta$ by $\theta_{p}$ for simplicity.

Multiplying equation (\ref{b0}) with $-(\Delta\theta_{p}+\partial_{p}^{2}\theta_{p})$ and integrating on domain $\mathcal{M}$, we have
 \begin{align*}
&-\int_{\mathcal{M}}\partial_{t}\theta_{p}(\Delta\theta_{p}+\partial_{p}^{2}\theta_{p})d\mathcal{M}
-\int_{\mathcal{M}}\partial_{p}\mathcal{A}_{\theta}\theta(\Delta\theta_{p}+\partial_{p}^{2}\theta_{p})d\mathcal{M}\nonumber\\
=&\int_{\mathcal{M}}[\textbf{\emph{v}}_{p}\cdot\nabla\theta+\textbf{\emph{v}}\cdot\nabla\theta_{p}
+(\nabla\cdot\textbf{\emph{v}})\theta_{p}+\omega\partial_{p}\theta_{p}-\partial_{p}f_{\theta}](\Delta\theta_{p}+\partial_{p}^{2}\theta_{p})d\mathcal{M}.
\end{align*}
It follows from integration by parts that
\begin{align*}
-\int_{\mathcal{M}}\partial_{t}\theta_{p}(\Delta\theta_{p}+\partial_{p}^{2}\theta_{p})d\mathcal{M}
=\frac{1}{2}\frac{d}{dt}(\|\nabla\theta_{p}\|_{L^{2}}^{2}+\|\partial_{p}\theta_{p}\|_{L^{2}}^{2}).
\end{align*}
Through a direct calculation and using the H\"older inequality and the Young inequality, we can infer that
\begin{align*}
-\int_{\mathcal{M}}\partial_{p}\mathcal{A}_{\theta}\theta\Delta\theta_{p}d\mathcal{M}
\geq\frac{3\mu_{2}}{4}\|\theta_{p}\|_{H^{2}}^{2}
-C\left(\|\nabla\theta\|_{L^{2}}^{2}+\|\nabla\theta_{p}\|_{L^{2}}^{2}+\|\theta_{p}\|_{L^{2}}^{2}+\|\partial_{p}\theta_{p}\|_{L^{2}}^{2}\right),
\end{align*}
where $\mu_{2}=(p_{0}/p_{1})^{2R/c_{p}}\rm{min}\{\mu_{\theta},\nu_{\theta}\}$.
Using the H\"older inequality and the Young inequality, we have
\begin{align*}
&\int_{\mathcal{M}}\textbf{\emph{v}}_{p}\cdot\nabla\theta(\Delta\theta_{p}+\partial_{p}^{2}\theta_{p})d\mathcal{M}\nonumber\\
\leq&\int_{\mathcal{M}}|\textbf{\emph{v}}_{p}||\nabla\theta||\Delta\theta_{p}+\partial_{p}^{2}\theta_{p}|d\mathcal{M}
\leq\|\textbf{\emph{v}}_{p}\|_{L^{6}}\|\nabla\theta\|_{L^{3}}\|\theta_{p}\|_{H^{2}}\nonumber\\
\leq& C\|\textbf{\emph{v}}_{p}\|_{H^{1}}\|\nabla\theta\|_{H^{1}}\|\Delta\theta_{p}\|_{L^{2}}
\leq C\|\textbf{\emph{v}}_{p}\|_{H^{1}}^{2}\|\theta\|_{H^{2}}^{2}
+\frac{\mu_{2}}{20}\|\theta_{p}\|_{H^{2}}^{2}.
\end{align*}
Considering the embedding relation $H^{2}\subset L^{\infty}$ in $\mathbb{R}^{3}$ and using the H\"older inequality, we can obtain that
\begin{align*}
\int_{\mathcal{M}}\textbf{\emph{v}}\cdot\nabla\theta_{p}(\Delta\theta_{p}+\partial_{p}^{2}\theta_{p})d\mathcal{M}
\leq\|\textbf{\emph{v}}\|_{L^{\infty}}\|\nabla\theta_{p}\|_{L^{2}}\|\theta_{p}\|_{H^{2}}
\leq C\|\textbf{\emph{v}}\|_{H^{2}}^{2}\|\theta_{p}\|_{H^{1}}^{2}+\frac{\mu_{2}}{20}\|\theta_{p}\|_{H^{2}}^{2}.
\end{align*}
Similarly,
\begin{align*}
\int_{\mathcal{M}}(\nabla\cdot\textbf{\emph{v}})\theta_{p}(\Delta\theta_{p}+\partial_{p}^{2}\theta_{p})d\mathcal{M}
\leq\|\nabla\textbf{\emph{v}}\|_{L^{6}}\|\theta_{p}\|_{L^{3}}\|\theta_{p}\|_{H^{2}}
\leq C\|\textbf{\emph{v}}\|_{H^{2}}^{2}\|\theta_{p}\|_{H^{1}}^{2}+\frac{\mu_{2}}{20}\|\theta_{p}\|_{H^{2}}^{2}.
\end{align*}
Applying the inequality in Lemma \ref{trilinear} and Young's inequality, we have
\begin{align*}
&\int_{\mathcal{M}}\omega\partial_{p}\theta_{p}(\Delta\theta_{p}+\partial_{p}^{2}\theta_{p})d\mathcal{M}
\leq\int_{\mathcal{M'}}\int_{p_{0}}^{p_{1}}|\nabla\textbf{\emph{v}}|dp\int_{p_{0}}^{p_{1}}|\partial_{p}\theta_{p}||\Delta\theta_{p}+\partial_{p}^{2}\theta_{p}|dpd\mathcal{M'}\nonumber\\
&\leq C\|\textbf{\emph{v}}\|_{H^{1}}^{\frac{1}{2}}\|\textbf{\emph{v}}\|_{H^{2}}^{\frac{1}{2}}
\|\partial_{p}\theta_{p}\|_{L^{2}}^{\frac{1}{2}}\|\partial_{p}\theta_{p}\|_{H^{1}}^{\frac{1}{2}}
\|\theta_{p}\|_{H^{2}}
\leq C\|\textbf{\emph{v}}\|_{H^{1}}^{\frac{1}{2}}\|\textbf{\emph{v}}\|_{H^{2}}^{\frac{1}{2}}
\|\partial_{p}\theta_{p}\|_{L^{2}}^{\frac{1}{2}}\|\theta_{p}\|_{H^{2}}^{\frac{3}{2}}\nonumber\\
&\leq C\|\textbf{\emph{v}}\|_{H^{1}}^{2}\|\textbf{\emph{v}}\|_{H^{2}}^{2}
\|\theta_{p}\|_{H^{1}}^{2}+\frac{\mu_{2}}{20}\|\theta_{p}\|_{H^{2}}^{2}.
\end{align*}
Utilizing the H\"older inequality again, we infer that
\begin{align*}
-\int_{\mathcal{M}}\partial_{p}f_{\theta}(\Delta\theta_{p}+\partial_{p}^{2}\theta_{p})d\mathcal{M}\leq C\|\partial_{p}f_{\theta}\|_{L^{2}}^{2}+\frac{\mu_{2}}{20}\|\theta_{p}\|_{H^{2}}^{2}.
\end{align*}
Combining the above inequalities, we have
\begin{align}\label{theta-p-H1}
\frac{d}{dt}\|\theta_{p}\|_{H^{1}}^{2}+\mu_{2}\|\theta_{p}\|_{H^{2}}^{2}
\leq& C\left(1+\|\textbf{\emph{v}}\|_{H^{2}}^{2}+\|\textbf{\emph{v}}\|_{H^{1}}^{2}\|\textbf{\emph{v}}\|_{H^{2}}^{2}\right)
\|\theta_{p}\|_{H^{1}}^{2}+\nonumber\\
&C\left(\|\theta\|_{H^{1}}^{2}+\|\textbf{\emph{v}}_{p}\|_{H^{1}}^{2}\|\theta\|_{H^{2}}^{2}
+\|\partial_{p}f_{\theta}\|_{L^{2}}^{2}\right).
\end{align}

Using the Gronwall inequality and regularities of $\textbf{\emph{v}},\theta$ in Lemma \ref{H1 strong existence} and Proposition \ref{pro1}, we can finish the proof.
\end{proof}
\end{proposition}
\begin{proposition}\label{pro3}
Let $\textbf{v}_{0}, \partial_{p}\textbf{v}_{0}\in\mathbb{V}$, $q_{0},\partial_{p}q_{0}\in H^{1},$ $\partial_{p}f_{\textbf{v}},\partial_{p}f_{q}\in L^{2}$. Then for any $t_{1}\in (0,\infty)$, the solution $q$ to equation (\ref{e5}) satisfies that
\begin{align*}
\partial_{p}q\in L^{\infty}\left(0,t_{1};H^{1}\right)\cap L^{2}\left(0,t_{1};H^{2}\right).
\end{align*}
\end{proposition}
\begin{proof}
The prove process of this proposition is almost the same to the calculation in dealing with $\theta$. The only thing that needs to be calculated separately is the estimate of
   \begin{align*}
   \int_{\mathcal{M}}\partial_{p}\mathcal{A}_{q}q
(-\Delta\partial_{p}q-\partial_{p}^{3}q)d\mathcal{M}.
\end{align*}

Through a direct calculation and integration by parts, we have
 \begin{align*}
 -\int_{\mathcal{M}}\partial_{p}\mathcal{A}_{q}q\Delta\partial_{p}qd\mathcal{M}
  =&\int_{\mathcal{M}}\partial_{p}\left(\mu_{q}\Delta q+\nu_{q}\partial_{p}\left(\left(\frac{gp}{R\bar{\theta}}\right)^{2}\partial_{p}q\right)\right)\Delta\partial_{p}qd\mathcal{M}\nonumber\\
 =&\mu_{q}\int_{\mathcal{M}}\Delta\partial_{p}q\Delta\partial_{p}qd\mathcal{M}+
 \nu_{q}\int_{\mathcal{M}}\partial_{p}\left(\partial_{p}\left(\left(\frac{gp}{R\bar{\theta}}\right)^{2}\partial_{p}q\right)\right)\Delta\partial_{p}qd\mathcal{M}\nonumber\\
 =&\mu_{q}\|\Delta\partial_{p}q\|_{L^{2}}^{2}+\nu_{q}\|\nabla\partial_{p}^{2}q\|_{w}^{2}
 +\nu_{q}\int_{\mathcal{M}}2\left(\frac{g}{R\bar{\theta}}\right)^{2}p\nabla\partial_{p}q\cdot\nabla\partial_{p}^{2}qd\mathcal{M}.
\end{align*}
Using the H\"older inequality, we can deduce that
\begin{align*}
\int_{\mathcal{M}}\left(\frac{g}{R\bar{\theta}}\right)^{2}p\nabla\partial_{p}q\cdot\nabla\partial_{p}^{2}qd\mathcal{M}
\leq C\|\nabla\partial_{p}q\|_{L^{2}}\|\nabla\partial_{p}^{2}q\|_{L^{2}}
\leq C\|\nabla\partial_{p}q\|_{L^{2}}^{2}+\frac{1}{4}\|\nabla\partial_{p}^{2}q\|_{w}^{2},
\end{align*}
where we have used the equivalence between $\|\cdot\|_{L^{2}}$ and $\|\cdot\|_{w}$.

Thus
\begin{align*}
 -\int_{\mathcal{M}}\partial_{p}\mathcal{A}_{q}q\Delta\partial_{p}qd\mathcal{M}
  \geq
 \mu_{q}\|\Delta\partial_{p}q\|_{L^{2}}^{2}+\frac{3\nu_{q}}{4}\|\nabla\partial_{p}^{2}q\|_{w}^{2}
 -C\|q\|_{H^{2}}^{2}.
\end{align*}

Through similar calculations, we can deal with
$-\int_{\mathcal{M}}\partial_{p}\mathcal{A}_{q}q\partial_{p}^{3}qd\mathcal{M}$.
So we omit the detail.
\end{proof}

\textbf{Proof of Theorem 2.3}
Combining the known well-posedness result in Lemma \ref{H1 strong existence} and the a priori estimates in Propositions \ref{pro1}, \ref{pro2} and \ref{pro3}, we can finish the proof.

\section{The well posedness of $H^{2}$ solution}
In this section, we consider the $H^{2}$ regularity of solutions. Due to the result in Theorem \ref{theorem1}, we only need to improve the regularity of solutions in the horizontal direction. We first consider the $H^{2}$ a priori estimate of $\textbf{\emph{v}}$.
\begin{proposition}\label{pro4}
Let $\textbf{v}_{0}\in\mathbb{V}\cap (H^{2})^{2}$, $f_{\textbf{v}}\in H^{1}$. Then for any $t_{1}\in (0,\infty)$, the solution $\textbf{v}$ to equation (\ref{e1}) satisfies that
\begin{align*}
\textbf{v}\in L^{\infty}\left(0,t_{1};\mathbb{V}\cap (H^{2})^{2}\right)\cap L^{2}\left(0,t_{1};(H^{3})^{2}\right).
\end{align*}
\end{proposition}
\begin{proof}
Multiplying equation (\ref{e1}) with $\Delta^{2}\textbf{\emph{v}}$ and integrating on domain $\mathcal{M}$, we have
\begin{align*}
\int_{\mathcal{M}}\partial_{t}\textbf{\emph{v}}\cdot\Delta^{2}\textbf{\emph{v}}d\mathcal{M}
+\int_{\mathcal{M}}\mathcal{A}_{\textbf{\emph{v}}}\textbf{\emph{v}}\cdot\Delta^{2}\textbf{\emph{v}}d\mathcal{M}
=&-\int_{\mathcal{M}}\left(\textbf{\emph{v}}\cdot\nabla\textbf{\emph{v}}+\omega\partial_{p}\textbf{\emph{v}}\right)\cdot\Delta^{2}\textbf{\emph{v}}d\mathcal{M}
\nonumber\\
&-\int_{\mathcal{M}}(f\textbf{\emph{v}}^{\bot}+\nabla\Phi-f_{\textbf{\emph{v}}})\cdot\Delta^{2}\textbf{\emph{v}}d\mathcal{M}.
\end{align*}
 Integrating by parts, we can infer that
\begin{align*}
\int_{\mathcal{M}}\partial_{t}\textbf{\emph{v}}\cdot\Delta^{2}\textbf{\emph{v}}d\mathcal{M}
=\frac{1}{2}\frac{d}{dt}\|\Delta\textbf{\emph{v}}\|_{L^{2}}^{2},
\end{align*}
and
\begin{align*}
\int_{\mathcal{M}}\mathcal{A}_{\textbf{\emph{v}}}\textbf{\emph{v}}\cdot\Delta^{2}\textbf{\emph{v}}d\mathcal{M}
=\mu_{v}\|\nabla\Delta\textbf{\emph{v}}\|_{L^{2}}^{2}+
\nu_{v}\|\Delta\partial_{p}\textbf{\emph{v}}\|_{w}^{2}.
\end{align*}
Obviously,
\begin{align*}
\int_{\mathcal{M}}f\textbf{\emph{v}}^{\bot}\cdot\Delta^{2}\textbf{\emph{v}}d\mathcal{M}=0.
\end{align*}
Considering equation (\ref{e2}), we know that
\begin{align*}
\Phi=-\int_{p_{1}}^{p}\frac{R}{p}Tdp+\Phi_{s},
\end{align*}
where $\Phi_{s}$ is a given geopotential at $p=p_{1}$. Then we can deduce that
\begin{align*}
&\int_{\mathcal{M}}\nabla\Phi\cdot\Delta^{2}\textbf{\emph{v}}d\mathcal{M}
=\int_{\mathcal{M}}\nabla\Phi_{s}\cdot\Delta^{2}\textbf{\emph{v}}d\mathcal{M}+
\int_{\mathcal{M}}\int_{p_{1}}^{p}\frac{-R\nabla T}{p'}dp'\cdot\Delta^{2}\textbf{\emph{v}}d\mathcal{M}\nonumber\\
&\leq \|\Delta\theta\|_{L^{2}}\|\nabla\Delta\textbf{\emph{v}}\|_{L^{2}}
\leq C\|\Delta\theta\|_{L^{2}}^{2}+\frac{\mu_{1}}{6}\|\nabla\Delta\textbf{\emph{v}}\|_{L^{2}}^{2},
\end{align*}
where we have used
\begin{align*}
\int_{\mathcal{M}}\nabla\Phi_{s}\cdot\Delta^{2}\textbf{\emph{v}}d\mathcal{M}
=\int_{\mathcal{M}'}\Phi_{s}\cdot\Delta^{2}(\nabla\cdot\bar{\textbf{\emph{v}}})d\mathcal{M}=0.
\end{align*}
Utilizing integration by parts, the H\"older inequality and the Young inequality, we have
 \begin{align*}
 \int_{\mathcal{M}}f_{\textbf{\emph{v}}}\cdot\Delta^{2}\textbf{\emph{v}}d\mathcal{M}
 \leq C\|f_{\textbf{\emph{v}}}\|_{H^{1}}^{2}+\frac{\mu_{1}}{6}\|\nabla\Delta\textbf{\emph{v}}\|_{L^{2}}.
 \end{align*}
 Through integration by parts, we infer that
 \begin{align*}
 &\int_{\mathcal{M}}\left(\textbf{\emph{v}}\cdot\nabla\textbf{\emph{v}}+\omega\partial_{p}\textbf{\emph{v}}\right)\cdot\Delta^{2}\textbf{\emph{v}}d\mathcal{M}\nonumber\\
 \leq&C\int_{\mathcal{M}}|\Delta\textbf{\emph{v}}|^{2}|\nabla\textbf{\emph{v}}|d\mathcal{M} +C\int_{\mathcal{M}}|\Delta\textbf{\emph{v}}||\Delta\omega||\partial_{p}\textbf{\emph{v}}|d\mathcal{M}
+C\int_{\mathcal{M}}|\Delta\textbf{\emph{v}}||\nabla\omega||\nabla\partial_{p}\textbf{\emph{v}}|d\mathcal{M},
 \end{align*}
 where we have used
 \begin{align*}
 \int_{\mathcal{M}}\left[\textbf{\emph{v}}\cdot\nabla(\Delta\textbf{\emph{v}})+\omega\partial_{p}(\Delta\textbf{\emph{v}})\right]\cdot\Delta\textbf{\emph{v}}d\mathcal{M}=0.
 \end{align*}
 Using the H\"older inequality and the Ladyzhenskaya inequality, we have
 \begin{align*}
 &\int_{\mathcal{M}}|\Delta\textbf{\emph{v}}|^{2}|\nabla\textbf{\emph{v}}|d\mathcal{M}
\leq C\|\nabla\textbf{\emph{v}}\|_{L^{6}}\|\Delta\textbf{\emph{v}}\|_{L^{3}}\|\Delta\textbf{\emph{v}}\|_{L^{2}}\nonumber\\
&\leq C\|\textbf{\emph{v}}\|_{H^{2}}\|\Delta\textbf{\emph{v}}\|_{L^{2}}^{\frac{3}{2}}
 \|\Delta\textbf{\emph{v}}\|_{H^{1}}^{\frac{1}{2}}
\leq C\|\textbf{\emph{v}}\|_{H^{2}}^{\frac{10}{3}}+\frac{\mu_{1}}{24}\|\textbf{\emph{v}}\|_{H^{3}}^{2}\nonumber\\
&\leq C(\|\textbf{\emph{v}}\|_{H^{2}}^{2}+1)\|\textbf{\emph{v}}\|_{H^{2}}^{2}+\frac{\mu_{1}}{18}\|\textbf{\emph{v}}\|_{H^{3}}^{2}.
  \end{align*}
    Using the inequality in Lemma \ref{trilinear}, we get
  \begin{align*}
  \int_{\mathcal{M}}|\Delta\textbf{\emph{v}}||\Delta\omega||\partial_{p}\textbf{\emph{v}}|d\mathcal{M}
  &\leq \int_{\mathcal{M}'}\int_{p_{0}}^{p_{1}}|\nabla\Delta\textbf{\emph{v}}|dp\int_{p_{0}}^{p_{1}}
  |\Delta\textbf{\emph{v}}||\partial_{p}\textbf{\emph{v}}|dpd\mathcal{M}'\nonumber\\
  &\leq
  C\|\nabla\Delta\textbf{\emph{v}}\|_{L^{2}}\|\Delta\textbf{\emph{v}}\|_{L^{2}}^{\frac{1}{2}}
  \|\Delta\textbf{\emph{v}}\|_{H^{1}}^{\frac{1}{2}}
  \|\textbf{\emph{v}}\|_{H^{1}}^{\frac{1}{2}}\|\partial_{p}\textbf{\emph{v}}\|_{H^{1}}^{\frac{1}{2}}\nonumber\\
 &\leq C\|\textbf{\emph{v}}\|_{H^{2}}^{\frac{1}{2}}\|\textbf{\emph{v}}\|_{H^{3}}^{\frac{3}{2}}
\|\textbf{\emph{v}}\|_{H^{1}}^{\frac{1}{2}}\|\partial_{p}\textbf{\emph{v}}\|_{H^{1}}^{\frac{1}{2}}\nonumber\\
 &\leq C\|\textbf{\emph{v}}\|_{H^{1}}^{2}\|\partial_{p}\textbf{\emph{v}}\|_{H^{1}}^{2}\|\textbf{\emph{v}}\|_{H^{2}}^{2}
 +\frac{\mu_{1}}{18}\|\textbf{\emph{v}}\|_{H^{3}}^{2},
 \end{align*}
 Using the H\"older inequality and Ladyzhenskaya's inequality again, we can deduce that
 \begin{align*}
 &\int_{\mathcal{M}}|\Delta\textbf{\emph{v}}||\nabla\omega||\nabla\partial_{p}\textbf{\emph{v}}|d\mathcal{M}
\leq C\|\Delta\textbf{\emph{v}}\|_{L^{2}}\|\nabla\omega\|_{L^{3}}
\|\nabla\partial_{p}\textbf{\emph{v}}\|_{L^{6}}\nonumber\\
&\leq C\|\Delta\textbf{\emph{v}}\|_{L^{2}}\|\nabla\omega\|_{L^{2}}^{\frac{1}{2}}
\|\nabla\omega\|_{H^{1}}^{\frac{1}{2}}
\|\partial_{p}\textbf{\emph{v}}\|_{H^{2}}
\leq C\|\Delta\textbf{\emph{v}}\|_{L^{2}}^{\frac{3}{2}}
\|\Delta\textbf{\emph{v}}\|_{H^{1}}^{\frac{1}{2}}
\|\partial_{p}\textbf{\emph{v}}\|_{H^{2}}\nonumber\\
&\leq C\|\textbf{\emph{v}}\|_{H^{2}}^{\frac{3}{2}}
\|\textbf{\emph{v}}\|_{H^{3}}^{\frac{1}{2}}
\|\partial_{p}\textbf{\emph{v}}\|_{H^{2}}
\leq C\|\partial_{p}\textbf{\emph{v}}\|_{H^{2}}^{\frac{4}{3}}\|\textbf{\emph{v}}\|_{H^{2}}^{2}
+\frac{\mu_{1}}{18}\|\textbf{\emph{v}}\|_{H^{3}}^{2}\nonumber\\
&\leq C\left(\|\partial_{p}\textbf{\emph{v}}\|_{H^{2}}^{2}+1\right)\|\textbf{\emph{v}}\|_{H^{2}}^{2}
+\frac{\mu_{1}}{18}\|\textbf{\emph{v}}\|_{H^{3}}^{2},
\end{align*}
where  we have used the relation between $\omega$ and $\textbf{\emph{v}}$ as well as the Minkowski inequality in the fourth step. Thus
\begin{align*}
&\int_{\mathcal{M}}\left(\textbf{\emph{v}}\cdot\nabla\textbf{\emph{v}}+\omega\partial_{p}\textbf{\emph{v}}\right)\Delta^{2}\textbf{\emph{v}}d\mathcal{M}\nonumber\\
\leq& C\left(1+\|\textbf{\emph{v}}\|_{H^{2}}^{2}+\|\textbf{\emph{v}}\|_{H^{1}}^{2}\|\textbf{\emph{v}}_{p}\|_{H^{1}}^{2}+
\|\textbf{\emph{v}}_{p}\|_{H^{2}}^{2}\right)\|\textbf{\emph{v}}\|_{H^{2}}^{2}
+\frac{\mu_{1}}{6}\|\textbf{\emph{v}}\|_{H^{3}}^{2}.
 \end{align*}
 Combining all the above inequalities, we infer that
 \begin{align}\label{d7}
 \frac{d}{dt}\|\Delta\textbf{\emph{v}}\|_{L^{2}}^{2}+&\mu_{1}\left(\|\nabla\Delta\textbf{\emph{v}}\|_{L^{2}}^{2}+
\|\Delta\partial_{p}\textbf{\emph{v}}\|_{w}^{2}\right)
\leq C\left(\|\theta\|_{H^{2}}^{2}+\|S_{\textbf{\emph{v}}}\|_{H^{1}}^{2}\right)+\frac{\mu_{1}}{2}\|\textbf{\emph{v}}\|_{H^{3}}^{2}+\nonumber\\
&C\left(1+\|\textbf{\emph{v}}\|_{H^{2}}^{2}+\|\textbf{\emph{v}}\|_{H^{1}}^{2}\|\textbf{\emph{v}}_{p}\|_{H^{1}}^{2}+
\|\textbf{\emph{v}}_{p}\|_{H^{2}}^{2}\right)\|\textbf{\emph{v}}\|_{H^{2}}^{2}.
 \end{align}

Thus we can deduce from (\ref{008}) and (\ref{d7}) that
\begin{align*}
 \frac{d}{dt}\|\textbf{\emph{v}}\|_{H^{2}}^{2}+\mu_{1}\|\textbf{\emph{v}}\|_{H^{3}}^{2}
\leq&
C\left(1+\|\textbf{\emph{v}}\|_{H^{2}}^{2}+\|\textbf{\emph{v}}\|_{H^{1}}^{2}\|\textbf{\emph{v}}_{p}\|_{H^{1}}^{2}+
\|\textbf{\emph{v}}_{p}\|_{H^{2}}^{2}+\|\textbf{\emph{v}}\|_{H^1}^{2}\|\textbf{\emph{v}}\|_{H^2}^{2} \right)\|\textbf{\emph{v}}\|_{H^{2}}^{2}\nonumber\\
&+C\left(\|\textbf{\emph{v}}\|_{H^{1}}^{2}+\|\textbf{\emph{v}}\|_{H^{1}}^{4}+\|\theta\|_{H^{1}}^{2}+\|\theta\|_{H^{2}}^{2}+\|f_{\textbf{\emph{v}}}\|_{H^{1}}^{2}\right).
 \end{align*}
Considering regularities of $\textbf{\emph{v}},\theta$ in Lemma \ref{H1 strong existence} and Theorem \ref{theorem1}, using the Gronwall inequality, we can finish the proof of this proposition.
\end{proof}
\begin{proposition}\label{pro5}
Let $\textbf{v}_{0}\in\mathbb{V}\cap (H^{2})^{2},\theta_{0},q_{0}\in H^{2}$, $f_{\textbf{v}},f_{\theta},f_{q}\in H^{1}$. Then for any $t_{1}\in (0,\infty)$, the solution $(\theta,q)$ to equations (\ref{e4})-(\ref{e5}) satisfies that
\begin{align*}
(\theta,q)\in L^{\infty}\left(0,t_{1}; (H^{2})^{2}\right)\cap L^{2}\left(0,t_{1};(H^{3})^{2}\right).
\end{align*}
\end{proposition}
\begin{proof}
Since the $H^{2}$ regularity of the solution in vertical direction has been obtained, we mainly need to consider the a priori estimates in horizontal direction.

For $\theta$, multiplying equation (\ref{e4}) with $\Delta^{2}\theta$ and integrating on domain $\mathcal{M}$, we have
\begin{align*}
\int_{\mathcal{M}}\partial_{t}\theta\Delta^{2}\theta d\mathcal{M} +\int_{\mathcal{M}}\mathcal{A}_{\theta}\theta\Delta^{2}\theta d\mathcal{M}
+\int_{\mathcal{M}}(\textbf{\emph{v}}\cdot\nabla\theta+\omega\partial_{p}\theta)\Delta^{2}\theta d\mathcal{M}
=\int_{\mathcal{M}}f_{\theta}\Delta^{2}\theta d\mathcal{M}.
\end{align*}
Through integration by parts, we can deduce that
\begin{align*}
\int_{\mathcal{M}}\partial_{t}\theta\Delta^{2}\theta d\mathcal{M}=\frac{1}{2}\frac{d}{dt}\|\Delta\theta\|_{L^{2}}^{2}.
\end{align*}
Through a direct calculation, using the H\"older inequality and the Young inequality, we can deduce that
\begin{align*}
\int_{\mathcal{M}}\mathcal{A}_{\theta}\theta\Delta^{2}\theta d\mathcal{M}
\geq\frac{3\mu_{2}}{4}\left(\|\nabla\Delta\theta\|_{L^{2}}^{2}+\|\Delta\partial_{p}\theta\|_{L^{2}}^{2}\right)
-C\left(\|\nabla\theta\|_{L^{2}}^{2}+\|\nabla\partial_{p}\theta\|_{L^{2}}^{2}+
\|\Delta\theta\|_{L^{2}}^{2} \right).
\end{align*}
By integration by parts, we have
\begin{align*}
\int_{\mathcal{M}}(\textbf{\emph{v}}\cdot\nabla\theta+\omega\partial_{p}\theta)\Delta^{2}\theta d\mathcal{M}
\leq &\int_{\mathcal{M}}|\nabla\textbf{\emph{v}}||\nabla\theta||\nabla\Delta\theta|d\mathcal{M} +\int_{\mathcal{M}}|\textbf{\emph{v}}||\Delta\theta||\nabla\Delta\theta|d\mathcal{M}\nonumber\\
&+\int_{\mathcal{M}}|\nabla\omega||\partial_{p}\theta||\nabla\Delta\theta|d\mathcal{M}
+\int_{\mathcal{M}}|\omega||\nabla\partial_{p}\theta||\nabla\Delta\theta|d\mathcal{M}.
\end{align*}
Then we estimate these four terms respectively.
Using the H\"older inequality and the Young inequality, we obtain that
\begin{align*}
\int_{\mathcal{M}}|\nabla\textbf{\emph{v}}||\nabla\theta||\nabla\Delta\theta|d\mathcal{M} &\leq\|\nabla\textbf{\emph{v}}\|_{L^{6}}\|\nabla\theta\|_{L^{3}}\|\nabla\Delta\theta\|_{L^{2}} \leq C\|\textbf{\emph{v}}\|_{H^{2}}\|\theta\|_{H^{2}}\|\nabla\Delta\theta\|_{L^{2}}\nonumber\\ &\leq C\|\textbf{\emph{v}}\|_{H^{2}}^{2}\|\theta\|_{H^{2}}^{2}+\frac{\mu_{2}}{32}\|\nabla\Delta\theta\|_{L^{2}}^{2},
\end{align*}
and
\begin{align*}
\int_{\mathcal{M}}|\textbf{\emph{v}}||\Delta\theta||\nabla\Delta\theta|d\mathcal{M} \leq&\|\textbf{\emph{v}}\|_{L^{\infty}}\|\Delta\theta\|_{L^{2}}\|\nabla\Delta\theta\|_{L^{2}} \leq C\|\textbf{\emph{v}}\|_{H^{2}}\|\theta\|_{H^{2}}\|\nabla\Delta\theta\|_{L^{2}}\nonumber\\ \leq &C\|\textbf{\emph{v}}\|_{H^{2}}^{2}\|\theta\|_{H^{2}}^{2}+\frac{\mu_{2}}{32}\|\nabla\Delta\theta\|_{L^{2}}^{2},
\end{align*}
where in the second step we have used the embedding relation $H^{2}\subset L^{\infty}$ in $\mathbb{R}^{3}$ .
Considering (\ref{e2}) and the Minkowski inequality in integral form, we have
\begin{align*}
\int_{\mathcal{M}}|\nabla\omega||\partial_{p}\theta||\nabla\Delta\theta|d\mathcal{M}
\leq\|\nabla\omega\|_{L^{6}}\|\partial_{p}\theta\|_{L^{3}}\|\nabla\Delta\theta\|_{L^{2}} \leq C\|\textbf{\emph{v}}\|_{H^{3}}^{2}\|\theta\|_{H^{2}}^{2}+\frac{\mu_{2}}{32}\|\nabla\Delta\theta\|_{L^{2}}^{2},
\end{align*}
and
\begin{align*}
\int_{\mathcal{M}}|\omega||\nabla\partial_{p}\theta||\nabla\Delta\theta|d\mathcal{M}
\leq&\|\omega\|_{L^{\infty}}\|\nabla\partial_{p}\theta\|_{L^{2}}\|\nabla\Delta\theta\|_{L^{2}} \leq C\|\textbf{\emph{v}}\|_{H^{3}}\|\theta\|_{H^{2}}\|\nabla\Delta\theta\|_{L^{2}}\nonumber\\ \leq&C\|\textbf{\emph{v}}\|_{H^{3}}^{2}\|\theta\|_{H^{2}}^{2}+\frac{\mu_{2}}{32}\|\nabla\Delta\theta\|_{L^{2}}^{2}.
\end{align*}
Thus
\begin{align*}
\int_{\mathcal{M}}(\textbf{\emph{v}}\cdot\nabla\theta+\omega\partial_{p}\theta)\Delta^{2}\theta d\mathcal{M}
\leq C\left(\|\textbf{\emph{v}}\|_{H^{2}}^{2}+\|\textbf{\emph{v}}\|_{H^{3}}^{2}\right)\|\theta\|_{H^{2}}^{2}
+\frac{\mu_{2}}{8}\|\nabla\Delta\theta\|_{L^{2}}^{2}.
\end{align*}
Integrating by parts and considering the H\"older inequality, we infer that
\begin{align*}
\int_{\mathcal{M}}f_{\theta}\Delta^{2}\theta d\mathcal{M}\leq C\|f_{\theta}\|_{H^{1}}^{2}+\frac{\mu_{2}}{8}\|\nabla\Delta\theta\|_{L^{2}}^{2}.
\end{align*}
Then combining all the above inequalities, we have
\begin{align}\label{f11}
\frac{1}{2}\frac{d}{dt}\|\Delta\theta\|_{L^{2}}^{2}+&\mu_{2}\frac{3}{4}\left(\|\nabla\Delta\theta\|_{L^{2}}^{2}
+\|\Delta\partial_{p}\theta\|_{L^{2}}^{2}\right)
\leq C\left(\|\textbf{\emph{v}}\|_{H^{2}}^{2}+\|\textbf{\emph{v}}\|_{H^{3}}^{2}\right)\|\theta\|_{H^{2}}^{2}+\nonumber\\ &C\left(\|\nabla\theta\|_{L^{2}}^{2}+\|\nabla\partial_{p}\theta\|_{L^{2}}^{2}+
\|\Delta\theta\|_{L^{2}}^{2}+\|f_{\theta}\|_{H^{1}}^{2}\right)+\frac{\mu_{2}}{4}\|\theta\|_{H^{3}}^{2}.
\end{align}
Thus we can deduce from (\ref{theta-p-H1}) and (\ref{f11}) that
\begin{align*}
\frac{d}{dt}\|\theta\|_{H^{2}}^{2}+\mu_{2}\|\theta\|_{H^{3}}^{2}
\leq& C\left(1+\|\textbf{\emph{v}}\|_{H^{2}}^{2}+\|\textbf{\emph{v}}\|_{H^{3}}^{2}+\|\textbf{\emph{v}}\|_{H^{1}}^{2}\|\textbf{\emph{v}}\|_{H^{2}}^{2}\right)
\|\theta\|_{H^{2}}^{2}\nonumber\\
&+C\left(\|\theta\|_{H^{1}}^{2}+\|\textbf{\emph{v}}_{p}\|_{H^{1}}^{2}\|\theta\|_{H^{2}}^{2}
+\|\nabla\partial_{p}\theta\|_{L^{2}}^{2}+
\|\Delta\theta\|_{L^{2}}^{2}+\|f_{\theta}\|_{H^{1}}^{2}\right).
\end{align*}
Considering regularities of $\textbf{\emph{v}},\theta$ in Lemma \ref{H1 strong existence}, Theorem \ref{theorem1} and Proposition \ref{pro4}, using the Gronwall inequality, we can get the desired regularity of $\theta$.

For $q$, through similar arguments as $\textbf{\emph{v}}$ and $\theta$, we can get
\begin{align*}
q\in L^{\infty}\left(0,t_{1}; H^{2}\right)\cap L^{2}\left(0,t_{1};H^{3}\right)
\end{align*}
for any $t_{1}\in (0,\infty)$.
\end{proof}

At last, we consider the time regularity of solutions.
\begin{proposition}\label{pro6}
Let $\textbf{v}_{0}\in\mathbb{V}\cap (H^{2})^{2},\textbf{v}_{0},\theta_{0},q_{0}\in H^{2}$, $f_{\textbf{v}},f_{\theta},f_{q}\in H^{1}$. Then the solution $(\textbf{v},\theta,q)$ to equations (\ref{e1})-(\ref{e5}) satisfies that
\begin{align*}
(\partial_{t}\textbf{v},\partial_{t}\theta,\partial_{t}q)\in L^{2}\left(0,t_{1}; \mathbb{V}\times(H^{1})^{2}\right)
\end{align*}
for any $t_{1}\in (0,\infty)$.
\end{proposition}
\begin{proof}
In order to check the time regularity, we start with the expression of $\partial_{t}\textbf{\emph{v}}$. We can deduce from (\ref{e1}) that
\begin{equation}\label{t-e1}
\partial_{t}\textbf{\emph{v}}=-\mathcal{A}_{\textbf{\emph{v}}}\textbf{\emph{v}}
-(\textbf{\emph{v}}\cdot\nabla)\textbf{\emph{v}}-w\partial_{p}\textbf{\emph{v}}-\nabla\Phi-f \textbf{\emph{v}}^{\bot}+f_{\textbf{\emph{v}}},
\end{equation}
Considering regularities of $\textbf{\emph{v}},\theta$, it is obviously that
\begin{align}
f\textbf{\emph{v}}^{\bot},f_{\textbf{\emph{v}}},\nabla\Phi,\mathcal{A}_{\textbf{\emph{v}}}\textbf{\emph{v}}\in L^{2}(0,t_{1};H^{1}).
\end{align}
Then we verify the most complex term $(\textbf{\emph{v}}\cdot\nabla)\textbf{\emph{v}}+\omega\partial_{p}\textbf{\emph{v}}$.
In fact
\begin{align*}
&\|(\textbf{\emph{v}}\cdot\nabla)\textbf{\emph{v}}\|_{L^{2}(0,t_{1};H^{1})}^{2}
=\int_{0}^{t_{1}}\|(\textbf{\emph{v}}\cdot\nabla)\textbf{\emph{v}}\|_{H^{1}}^{2}ds\nonumber\\
&\leq\int_{0}^{t_{1}}\|(\textbf{\emph{v}}\cdot\nabla)\textbf{\emph{v}}\|_{L^{2}}^{2}ds
+\int_{0}^{t_{1}}\|\nabla\left((\textbf{\emph{v}}\cdot\nabla)\textbf{\emph{v}}\right)\|_{L^{2}}^{2}ds
+\int_{0}^{t_{1}}\|\partial_{p}\left((\textbf{\emph{v}}\cdot\nabla)\textbf{\emph{v}}\right)\|_{L^{2}}^{2}ds.
\end{align*}
Utilizing the H\"older inequality and the embedding relation $H^{2}\subset L^{\infty}$, considering the regularity results in Lemma \ref{H1 strong existence} and Proposition \ref{pro4}, we can deduce that
\begin{align*}
\int_{0}^{t_{1}}\|(\textbf{\emph{v}}\cdot\nabla)\textbf{\emph{v}}\|_{L^{2}}^{2}ds
\leq\int_{0}^{t_{1}}\|\textbf{\emph{v}}\|_{L^{\infty}}^{2}\|\nabla\textbf{\emph{v}}\|_{L^{2}}^{2}ds
\leq Ct_{1}\|\textbf{\emph{v}}\|_{L^{\infty}(0,t_{1};H^{2})}^{2}\|\textbf{\emph{v}}\|_{_{L^{\infty}(0,t_{1};H^{1})}}^{2}<\infty,
\end{align*}
\begin{align*}
&\int_{0}^{t_{1}}\|\nabla\left((\textbf{\emph{v}}\cdot\nabla)\textbf{\emph{v}}\right)\|_{L^{2}}^{2}ds
\leq\int_{0}^{t_{1}}\int_{
\mathcal{M}}|\nabla\textbf{\emph{v}}|^{4}d\mathcal{M}ds
+\int_{0}^{t_{1}}\int_{\mathcal{M}}|\textbf{\emph{v}}|^{2}|\Delta\textbf{\emph{v}}|^{2}d\mathcal{M}ds\nonumber\\
&\leq\int_{0}^{t_{1}}\|\nabla\textbf{\emph{v}}\|_{L^{\infty}}^{2}\|\nabla\textbf{\emph{v}}\|_{L^{2}}^{2}ds
+\int_{0}^{t_{1}}\|\textbf{\emph{v}}\|_{L^{\infty}}^{2}\|\Delta\textbf{\emph{v}}\|_{L^{2}}^{2}ds\nonumber\\
&\leq C\|\textbf{\emph{v}}\|_{L^{2}(0,t_{1};H^{3})}^{2}\|\textbf{\emph{v}}\|_{L^{\infty}(0,t_{1};H^{2})}^{2}
+Ct_{1}\|\textbf{\emph{v}}\|_{L^{\infty}(0,t_{1};H^{2})}^{4}<\infty,
\end{align*}
and
\begin{align*}
&\int_{0}^{t_{1}}\|\partial_{p}\left((\textbf{\emph{v}}\cdot\nabla)\textbf{\emph{v}}\right)\|_{L^{2}}^{2}ds
\leq\int_{0}^{t_{1}}\int_{\mathcal{M}}|\partial_{p}\textbf{\emph{v}}|^{2}|\nabla\textbf{\emph{v}}|^{2}d\mathcal{M}ds
+\int_{0}^{t_{1}}\int_{\mathcal{M}}|\textbf{\emph{v}}|^{2}|\nabla\partial_{p}\textbf{\emph{v}}|^{2}d\mathcal{M}ds\nonumber\\
&\leq\int_{0}^{t_{1}}\|\partial_{p}\textbf{\emph{v}}\|_{L^{\infty}}^{2}\|\nabla\textbf{\emph{v}}\|_{L^{2}}^{2}ds
+\int_{0}^{t_{1}}\|\textbf{\emph{v}}\|_{L^{\infty}}^{2}\|\nabla\partial_{p}\textbf{\emph{v}}\|_{L^{2}}^{2}ds\nonumber\\
&\leq C\|\textbf{\emph{v}}\|_{L^{2}(0,t_{1};H^{3})}^{2}\|\textbf{\emph{v}}\|_{L^{\infty}(0,t_{1};H^{1})}^{2}
+Ct_{1}\|\textbf{\emph{v}}\|_{L^{\infty}(0,t_{1};H^{2})}^{4}<\infty.
\end{align*}
Thus
\begin{align}
(\textbf{\emph{v}}\cdot\nabla)\textbf{\emph{v}}\in L^{2}(0,t_{1};H^{1}).
\end{align}
Similarly,
\begin{align*}
&\|\omega\partial_{p}\textbf{\emph{v}}\|_{L^{2}(0,t_{1};H^{1})}
\leq\int_{0}^{t_{1}}\|\omega\partial_{p}\textbf{\emph{v}}\|_{L^{2}}^{2}ds
+\int_{0}^{t_{1}}\|\nabla\left(\omega\partial_{p}\textbf{\emph{v}}\right)\|_{L^{2}}^{2}ds
+\int_{0}^{t_{1}}\|\partial_{p}\left(\omega\partial_{p}\textbf{\emph{v}}\right)\|_{L^{2}}^{2}ds.
\end{align*}
Considering the relation between $\omega$ and $\textbf{\emph{v}}$, we have
\begin{align*}
\int_{0}^{t_{1}}\|\omega\partial_{p}\textbf{\emph{v}}\|_{L^{2}}^{2}ds
\leq\int_{0}^{t_{1}}\|\partial_{p}\textbf{\emph{v}}\|_{L^{\infty}}^{2}\|\omega\|_{L^{2}}^{2}ds
\leq\|\textbf{\emph{v}}\|_{L^{2}(0,t_{1};H^{3})}^{2}\|\textbf{\emph{v}}\|_{L^{\infty}(0,t_{1};H^{1})}^{2}
<\infty,
\end{align*}
\begin{align*}
&\int_{0}^{t_{1}}\|\nabla\left(\omega\partial_{p}\textbf{\emph{v}}\right)\|_{L^{2}}^{2}ds
\leq\int_{0}^{t_{1}}\int_{\mathcal{M}}|\nabla\omega|^{2}|\partial_{p}\textbf{\emph{v}}|^{2}d\mathcal{M}ds
+\int_{0}^{t_{1}}\int_{\mathcal{M}}|\omega|^{2}|\nabla\partial_{p}\textbf{\emph{v}}|^{2}d\mathcal{M}ds\nonumber\\
&\leq\int_{0}^{t_{1}}\|\nabla\omega\|_{L^{2}}^{2}\|\partial_{p}\textbf{\emph{v}}\|_{L^{\infty}}^{2}ds
+\int_{0}^{t_{1}}\|\omega\|_{L^{\infty}}^{2}\|\nabla\partial_{p}\textbf{\emph{v}}\|_{L^{2}}^{2}ds\nonumber\\
&\leq\int_{0}^{t_{1}}\|\textbf{\emph{v}}\|_{H^{2}}^{2}\|\textbf{\emph{v}}\|_{H^{3}}^{2}ds
+\int_{0}^{t_{1}}\|\textbf{\emph{v}}\|_{H^{3}}^{2}\|\textbf{\emph{v}}\|_{H^{2}}^{2}ds\nonumber\\
&\leq\|\textbf{\emph{v}}\|_{L^{\infty}(0,t_{1};H^{2})}^{2}\|\textbf{\emph{v}}\|_{L^{2}(0,t_{1};H^{3})}^{2}<\infty,
\end{align*}
and
\begin{align*}
&\int_{0}^{t_{1}}\|\partial_{p}\left(\omega\partial_{p}\textbf{\emph{v}}\right)\|_{L^{2}}^{2}ds
\leq\int_{0}^{t_{1}}\int_{\mathcal{M}}|\partial_{p}\omega|^{2}|\partial_{p}\textbf{\emph{v}}|^{2}d\mathcal{M}ds
+\int_{0}^{t_{1}}\int_{\mathcal{M}}|\omega|^{2}|\partial_{p}^{2}\textbf{\emph{v}}|^{2}d\mathcal{M}ds\nonumber\\
&\leq\int_{0}^{t_{1}}\|\partial_{p}\omega\|_{L^{2}}^{2}\|\partial_{p}\textbf{\emph{v}}\|_{L^{\infty}}^{2}ds
+\int_{0}^{t_{1}}\|\omega\|_{L^{\infty}}^{2}\|\partial_{p}^{2}\textbf{\emph{v}}\|_{L^{2}}^{2}ds\nonumber\\
&\leq C\|\textbf{\emph{v}}\|_{L^{\infty}(0,t_{1};H^{1})}^{2}\|\textbf{\emph{v}}\|_{L^{2}(0,t_{1};H^{3})}^{2}
+C\|\textbf{\emph{v}}\|_{L^{\infty}(0,t_{1};H^{2})}^{2}\|\textbf{\emph{v}}\|_{L^{2}(0,t_{1};H^{3})}^{2}
<\infty.
\end{align*}
Thus
\begin{align*}
\|\omega\partial_{p}\textbf{\emph{v}}\|_{L^{2}(0,t_{1};H^{1})}<\infty.
\end{align*}
Combining all the above calculations, we can deduce that $\partial_{t}\textbf{\emph{v}}\in L^{2}(0,t_{1};H^{1})$.

Next, we consider the regularity of time derivation of $\theta$ and $q$. Here we only show the calculation about $q$. $\theta$ can be dealed with almost the same procedure. We can deduce from (\ref{e5}) that
\begin{equation*}
\partial_{t}q=-\mathcal{A}_{q}q-\textbf{\emph{v}}\cdot\nabla q-w\partial_{p}q+f_{q}.
\end{equation*}
It is obviously that $f_{q}\in L^{2}(0,t_{1};H^{1})$ and
\begin{align*}
\|\mathcal{A}_{q}q\|_{L^{2}(0,t_{1};H^{1})}\leq C\|q\|_{L^{2}(0,t_{1};H^{3})}<\infty.
\end{align*}
Using the definition of $\|\cdot\|_{H^{1}}$ norm, we have
\begin{align*}
&\|\textbf{\emph{v}}\cdot\nabla q\|_{L^{2}(0,t_{1};H^{1})}^{2}
=\int_{0}^{t_{1}}\|\textbf{\emph{v}}\cdot\nabla q\|_{H^{1}}^{2}ds\nonumber\\
&\leq\int_{0}^{t_{1}}\|\textbf{\emph{v}}\cdot\nabla q\|_{L^{2}}^{2}ds
+\int_{0}^{t_{1}}\|\nabla\left(\textbf{\emph{v}}\cdot\nabla q\right)\|_{L^{2}}^{2}ds
+\int_{0}^{t_{1}}\|\partial_{p}\left(\textbf{\emph{v}}\cdot\nabla q\right)\|_{L^{2}}^{2}ds.
\end{align*}
Using the H\"older inequality, we can obtain that
\begin{align*}
\int_{0}^{t_{1}}\|\textbf{\emph{v}}\cdot\nabla q\|_{L^{2}}^{2}ds
\leq\int_{0}^{t_{1}}\|\textbf{\emph{v}}\|_{L^{\infty}}^{2}\|\nabla q\|_{L^{2}}^{2}ds
\leq Ct_{1}\|\textbf{\emph{v}}\|_{L^{\infty}(0,t_{1};H^{2})}^{2}\| q\|_{L^{\infty}(0,t_{1};H^{1})}^{2}<\infty,
\end{align*}
\begin{align*}
&\int_{0}^{t_{1}}\|\nabla\left(\textbf{\emph{v}}\cdot\nabla q\right)\|_{L^{2}}^{2}ds
\leq\int_{0}^{t_{1}}\|\nabla\textbf{\emph{v}}\|_{L^{\infty}}^{2}\|\nabla q\|_{L^{2}}^{2}ds+\int_{0}^{t_{1}}\|\textbf{\emph{v}}\|_{L^{\infty}}^{2}\|\Delta q\|_{L^{2}}^{2}ds\nonumber\\
&\leq C\|\textbf{\emph{v}}\|_{L^{2}(0,t_{1};H^{3})}^{2}\| q\|_{L^{\infty}(0,t_{1};H^{1})}^{2}+Ct_{1}\|\textbf{\emph{v}}\|_{L^{\infty}(0,t_{1};H^{2})}^{2}\| q\|_{L^{\infty}(0,t_{1};H^{2})}^{2}<\infty,
\end{align*}
and
\begin{align*}
&\int_{0}^{t_{1}}\|\partial_{p}\left(\textbf{\emph{v}}\cdot\nabla q\right)\|_{L^{2}}^{2}ds
\leq\int_{0}^{t_{1}}\|\partial_{p}\textbf{\emph{v}}\|_{L^{\infty}}^{2}\|\nabla q\|_{L^{2}}^{2}ds+\int_{0}^{t_{1}}\|\textbf{\emph{v}}\|_{L^{\infty}}^{2}\|\nabla\partial_{p} q\|_{L^{2}}^{2}ds\nonumber\\
&\leq C\|\textbf{\emph{v}}\|_{L^{2}(0,t_{1};H^{3})}^{2}\| q\|_{L^{\infty}(0,t_{1};H^{1})}^{2}+Ct_{1}\|\textbf{\emph{v}}\|_{L^{\infty}(0,t_{1};H^{2})}^{2}\| q\|_{L^{\infty}(0,t_{1};H^{2})}^{2}<\infty.
\end{align*}
Thus we have $\textbf{\emph{v}}\cdot\nabla q\in L^{2}(0,t_{1};H^{1}).$
Similarly,
\begin{align*}
\|\omega\partial_{p}q\|_{L^{2}(0,t_{1};H^{1})}^{2}
\leq& C\|\textbf{\emph{v}}\|_{L^{2}(0,t_{1};H^{3})}^{2}\| q\|_{L^{\infty}(0,t_{1};H^{1})}^{2}+C\|\textbf{\emph{v}}\|_{L^{\infty}(0,t_{1};H^{2})}^{2}\| q\|_{L^{2}(0,t_{1};H^{3})}^{2}\nonumber\\
&+C\|\textbf{\emph{v}}\|_{L^{2}(0,t_{1};H^{3})}^{2}\| q\|_{L^{\infty}(0,t_{1};H^{2})}^{2}<\infty,
\end{align*}
where in the fifth step we have used the Minkowski inequality in integral form. Thus we can deduce that $q\in L^{2}(0,t_{1};H^{1})$.
\end{proof}

\textbf{Proof of Theorem 2.4}
Combining results in Propositions \ref{pro4}-\ref{pro6} and Theorem \ref{theorem1}, we can complete the proof.


\end{document}